\documentclass[12pt]{amsart}
\usepackage[T1]{fontenc}
\usepackage{lmodern}
\usepackage{amsmath,amssymb,amsthm}
\usepackage[a4paper,margin=30mm]{geometry}
\usepackage{microtype}
\usepackage[hidelinks]{hyperref}

\newtheorem{theorem}{Theorem}[section]
\newtheorem{proposition}[theorem]{Proposition}
\newtheorem{lemma}[theorem]{Lemma}
\newtheorem{corollary}[theorem]{Corollary}
\newtheorem{definition}[theorem]{Definition}

\newcommand{\C}{\mathbb C}

\title{Strongly linearly convex exhaustion of $\mathbb C$-convex domains with $C^{1,\alpha}$ boundary}
\author{Armen Edigarian}
\thanks{Funded by the National Science Centre, Poland under the Weave UNISONO, UMO-2025/07/Y/ST1/00146}

\begin{document}
\maketitle

\begin{abstract}
Let $D\subset\C^n$ be a bounded $\mathbb C$-convex domain with boundary of class $C^{1,\alpha}$, where $\alpha>1/2$. We prove that $D$ admits an increasing exhaustion by bounded $C^\infty$ strongly linearly convex domains. This gives, in the class $C^{1,\alpha}$, a positive answer to the approximation problem originally posed by Aizenberg \cite{Aiz94}.
\end{abstract}



\section{Introduction}

The approximation problem considered here goes back to Aizenberg \cite{Aiz94}:
\medskip
\noindent\emph{Does every bounded $\C$-convex domain admit an increasing exhaustion by smooth strongly linearly convex domains?}
\medskip

The first general positive result with strict smooth approximants is due to D. Jacquet. He proved that a bounded $\C$-convex domain with $C^2$ boundary admits an increasing exhaustion by $C^\infty$ strictly $\C$-convex domains \cite{Jac06}. D. Jacquet subsequently developed a $C^1$ theory in his thesis. Under a uniform interior ball condition he characterized $\C$-convexity near the boundary in terms of $-\log\delta_D^2$ and obtained an exhaustion by strictly $\C$-convex domains with $C^1$ boundary \cite[Theorems~3.2.2 and~3.2.4]{Jac08}. Thus his $C^2$ theorem gives smooth strongly linearly convex approximants, whereas his $C^1$ theorem requires an additional geometric hypothesis and does not smooth the approximating boundaries.

Pflug and Zwonek isolated the strong form of the approximation problem and asked whether every bounded $\C$-convex domain can be exhausted by strongly linearly convex domains \cite{PZ12}. They answered this affirmatively for the symmetrized bidisc by an explicit construction. Their result is especially relevant from the viewpoint of invariant distances, because the symmetrized bidisc is nonsmooth and nevertheless the strongly linearly convex exhaustion gives another route to the Lempert equality.

The purpose of the present paper is to lower Jacquet's $C^2$ regularity assumption while keeping smooth strongly linearly convex approximants and without imposing an interior ball condition.

\begin{theorem}\label{thm:main}
Let $D\subset\C^n$, $n\ge2$, be a bounded $\C$-convex domain with boundary of class $C^{1,\alpha}$, where $1/2<\alpha\le1$. Then there are bounded domains $D_j\subset\C^n$ such that
\[
 \overline{D_j}\subset D_{j+1},\qquad \bigcup_{j=1}^\infty D_j=D,
\]
and every $D_j$ has $C^\infty$ boundary and is strongly linearly convex.
\end{theorem}

The theorem has an immediate consequence for invariant distances and metrics \cite{JP,KW,Lem81,Lem84}. 

\begin{corollary}\label{cor:lempert}
Under the assumptions of Theorem~\ref{thm:main}, the Carath\'eodory distance equals the Lempert function on $D$, and the Carath\'eodory--Reiffen metric equals the Kobayashi--Royden metric.
\end{corollary}

We use the Hermitian product $\langle z,w\rangle=\sum_j z_j\overline{w_j}$, linear in the first variable. For a real $C^2$ function $u$ and $w\in\C^n$ set
\[
 \partial u(w)=\sum_{j=1}^n \frac{\partial u}{\partial z_j}w_j,
 \qquad
 L_u(w)=\sum_{j,k=1}^n \frac{\partial^2 u}{\partial z_j\partial\bar z_k}w_j\overline{w_k},
 \qquad
 Q_u(w)=\sum_{j,k=1}^n \frac{\partial^2u}{\partial{z_j}\partial z_k}w_jw_k.
\]
If $\zeta=t e^{i\theta}$, then differentiation of the one-variable function $\zeta\mapsto u(z+\zeta w)$ gives
\begin{equation}\label{eq:complex-line-second}
 \frac{d^2}{dt^2}u(z+t e^{i\theta}w)\bigg|_{t=0}
 =2L_u(z;w)+2\operatorname{Re}\bigl(e^{2i\theta}Q_u(z;w)\bigr).
\end{equation}
\begin{definition}
A domain $D\subset\C^n$ is $\C$-convex if its intersection with every affine complex line is either empty or connected and simply connected.
\end{definition}

A bounded domain $D\subset\C^n$ with $C^2$ boundary is \emph{strongly linearly convex} if it has a $C^2$ defining function $r$ such that, for every $p\in\partial D$ and every nonzero complex tangent vector $w$, that is, every $w\ne0$ satisfying $\partial r(p)(w)=0$,
\begin{equation}\label{eq:slc}
 L_r(p;w)>|Q_r(p;w)|.
\end{equation}
The condition is independent of the choice of defining function.

\section{Proof of Theorem~\ref{thm:main}}

The proof retains Jacquet's two central ideas: work with $h=\operatorname{dist}(\cdot,\partial D)^2$ and then replace the regularized $h$ by $h-\varepsilon(1+|z|^2)$. 

If $\alpha=1$, we may replace $\alpha$ by any number in $(1/2,1)$. Thus we assume $1/2<\alpha<1$ and put
\[
 \delta_D(z)=\operatorname{dist}(z,\partial D),\qquad
 h(z)=\delta_D(z)^2,\qquad
 \gamma=\frac1{1-\alpha}>2.
\]
Fix $z_*\in D$ and choose $R>0$ so that $\overline D\subset B(0,R)$. Let
$\chi\in C_c^\infty(\C^n)$ be nonnegative, supported in the Euclidean unit
ball, and normalized by $\int\chi=1$. For $\tau>0$ put
$\chi_\tau(z)=\tau^{-2n}\chi(z/\tau)$.

Fix $0<c<\frac1{20(1+R^2)}$.
For small $\delta>0$ set
$\tau_\delta=\delta^\gamma$, $h_\delta=h*\chi_{\tau_\delta}$
on $\{\delta_D>\tau_\delta\}$, and define
\[
 \widetilde h_\delta(z)=h_\delta(z)-c\delta^2(1+|z|^2).
\]
Finally put
\begin{equation}\label{eq:Vdelta}
 V_\delta=
 \left\{z\in D:\delta_D(z)>\tau_\delta,\ 
 \widetilde h_\delta(z)>\frac92\delta^2\right\}.
\end{equation}
Let $G_\delta$ be the connected component of $V_\delta$ containing $z_*$. Let us
prove that, for all sufficiently small $\delta$, $G_\delta$ has $C^\infty$ strongly
linearly convex boundary. Take
$\delta_j=2^{-j}\delta_0$, $D_j=G_{\delta_j}$,
with $\delta_0>0$ sufficiently small.

For $p\in\partial D$, let $\nu(p)$ be the outward Euclidean unit normal.
Since $D$ is $\C$-convex and $\partial D$ is $C^1$, the
complex tangent hyperplane
\[
 \{z\in\C^n:\langle z-p,\nu(p)\rangle=0\}
\]
is disjoint from $D$; see \cite{APS,NT}. Hence
\begin{equation}\label{eq:envelope}
 \delta_D(z)=
 \inf_{p\in\partial D}|\langle z-p,\nu(p)\rangle|,
 \qquad z\in D.
\end{equation}

The following result is closely related to Jacquet's $C^1$ supporting inequality; compare \cite[Lemma~3.1.4]{Jac08}.

\begin{proposition}\label{prop:second-difference}
Let $z\in D$ be a differentiability point of $h$. Then, for any
$\theta\in\mathbb R$, any $w\in\C^n$, and any real $t$ sufficiently small
that $z\pm te^{i\theta}w\in D$,
\begin{equation}\label{eq:second-difference}
 h(z+t e^{i\theta}w)+h(z-t e^{i\theta}w)-2h(z)
 \le 2t^2\frac{|\partial h(z)(w)|^2}{h(z)}.
\end{equation}
\end{proposition}

\begin{proof}
Choose $p\in\partial D$ at which the infimum in
\eqref{eq:envelope} is attained at $z$. Then
\[
 h(\zeta)\le |\langle \zeta-p,\nu(p)\rangle|^2,
 \qquad \zeta\in D,
\]
with equality at $\zeta=z$. Since $h$ is differentiable at $z$, we obtain
\[
 \partial h(z)(w)
 =\overline{\langle z-p,\nu(p)\rangle}\,
 \langle w,\nu(p)\rangle.
\]
Hence
\[
 \frac{|\partial h(z)(w)|^2}{h(z)}
 =|\langle w,\nu(p)\rangle|^2.
\]
Using the preceding majorant at $z\pm te^{i\theta}w$ gives
\[
\begin{aligned}
 &h(z+t e^{i\theta}w)+h(z-t e^{i\theta}w)-2h(z)\\
 &\quad\le
 |\langle z-p+t e^{i\theta}w,\nu(p)\rangle|^2
 +|\langle z-p-t e^{i\theta}w,\nu(p)\rangle|^2\\
 &\qquad -2|\langle z-p,\nu(p)\rangle|^2\\
 &\quad=2t^2|\langle w,\nu(p)\rangle|^2,
\end{aligned}
\]
which proves \eqref{eq:second-difference}.
\end{proof}

Since $\partial D$ is compact and of class $C^{1,\alpha}$, there is
$M>0$ such that
\[
 |\nu(p)-\nu(q)|\le M|p-q|^\alpha,
 \qquad p,q\in\partial D.
\]
At a differentiability point $z$ of $h$, if $p(z)$ is a nearest boundary
point, then
\[
 z=p(z)-\delta_D(z)\nu(p(z))
\]
and
\begin{equation}\label{eq:normalrepresentation}
 \partial h(z)(w)
 =\langle w,z-p(z)\rangle
 =-\delta_D(z)\langle w,\nu(p(z))\rangle.
\end{equation}

\begin{lemma}\label{lem:gradosc}
There is a constant $C>0$ such that, for all sufficiently small $\delta>0$,
whenever $z,\zeta$ are differentiability points of $h=\delta_D^2$ satisfying
\[
 \frac\delta2\le\delta_D(z),\delta_D(\zeta)\le5\delta,
 \qquad |z-\zeta|\le\frac\delta4,
\]
then, for every $w\in\C^n$,
\begin{equation}\label{eq:gradestimate}
 |\partial h(z)(w)-\partial h(\zeta)(w)|
 \le C\bigl(|z-\zeta|+\delta^\gamma\bigr)|w|.
\end{equation}
\end{lemma}

\begin{proof}
Put
\[
 u(z)=z-p(z)=-\delta_D(z)\nu(p(z)).
\]
Since $\delta_D$ is $1$-Lipschitz and $\delta_D(\zeta)\le5\delta$,
\[
\begin{aligned}
 |u(z)-u(\zeta)|
 &\le |\delta_D(z)-\delta_D(\zeta)|
 +\delta_D(\zeta)|\nu(p(z))-\nu(p(\zeta))|\\
 &\le |z-\zeta|+C\delta|p(z)-p(\zeta)|^\alpha.
\end{aligned}
\]
Since
\[
 p(z)-p(\zeta)=(z-\zeta)-(u(z)-u(\zeta)),
\]
we obtain
\[
 |u(z)-u(\zeta)|
 \le |z-\zeta|
 +C\delta\bigl(|z-\zeta|+|u(z)-u(\zeta)|\bigr)^\alpha.
\]
If
\[
 |z-\zeta|+|u(z)-u(\zeta)|\le4|z-\zeta|,
\]
then $|u(z)-u(\zeta)|\le4|z-\zeta|$. Otherwise the preceding inequality gives
\[
 \frac12\bigl(|z-\zeta|+|u(z)-u(\zeta)|\bigr)
 \le C\delta\bigl(|z-\zeta|+|u(z)-u(\zeta)|\bigr)^\alpha,
\]
and hence
\[
 |z-\zeta|+|u(z)-u(\zeta)|
 \le C\delta^{1/(1-\alpha)}=C\delta^\gamma.
\]
Thus
\[
 |u(z)-u(\zeta)|\le C\bigl(|z-\zeta|+\delta^\gamma\bigr).
\]
By \eqref{eq:normalrepresentation},
\[
 \partial h(z)(w)-\partial h(\zeta)(w)
 =\langle w,u(z)-u(\zeta)\rangle,
\]
and \eqref{eq:gradestimate} follows.
\end{proof}

Note that
\begin{equation}\label{eq:hosc}
 |h(z)-h(\zeta)|\le10\delta|z-\zeta|
\end{equation}
whenever $\delta_D(z),\delta_D(\zeta)\le5\delta$. Indeed,
\[
 |\delta_D(z)^2-\delta_D(\zeta)^2|
 \le(\delta_D(z)+\delta_D(\zeta))|\delta_D(z)-\delta_D(\zeta)|,
\]
and $\delta_D$ is $1$-Lipschitz.

For small $\delta>0$ set
\[
 U_\delta=\{z\in D:\delta<\delta_D(z)<4\delta\}.
\]
Since $\tau_\delta=\delta^\gamma<\delta/4$ for small $\delta$,
$h_\delta$ is well defined on a neighborhood of $U_\delta$.

If $u$ is locally Lipschitz and $B(z,\tau)\Subset\operatorname{dom}u$, then
\begin{equation}\label{eq:diffconv}
 \partial(u*\chi_\tau)(z)(w)
 =\int \partial u(\zeta)(w)\chi_\tau(z-\zeta)\,d\zeta.
\end{equation}
In a real coordinate direction $e$ a change of variables
gives
\[
 \frac{(u*\chi_\tau)(z+te)-(u*\chi_\tau)(z)}t
 =\int\frac{u(\zeta+te)-u(\zeta)}t\chi_\tau(z-\zeta)\,d\zeta.
\]
The quotient is bounded by the Lipschitz constant of $u$, and the
restriction of $u$ to almost every line is absolutely continuous.
Dominated convergence gives the derivative formula in the direction $e$;
combining the two real directions belonging to each complex coordinate
gives \eqref{eq:diffconv}.

\begin{proposition}\label{prop:mollification}
There is a constant $C>0$ such that, for all sufficiently small $\delta$,
every $z\in U_\delta$ and every $w\in\C^n$ satisfy
\begin{equation}\label{eq:mollifiedineq}
 \frac{|\partial h_\delta(z)(w)|^2}{h_\delta(z)}
 -L_{h_\delta}(z;w)-|Q_{h_\delta}(z;w)|
 \ge -C\delta^{\frac{2\alpha}{1-\alpha}}|w|^2.
\end{equation}
Moreover
\begin{equation}\label{eq:C0approx}
 |h_\delta(z)-h(z)|\le10\delta\tau_\delta,
 \qquad z\in U_\delta.
\end{equation}
\end{proposition}

\begin{proof}
Fix $z\in U_\delta$, $w\in\C^n$, and $\theta\in\mathbb R$. Since
$\delta_D(\zeta)\ge\delta-\tau_\delta$ for every
$\zeta\in B(z,\tau_\delta)$, there is $t_0>0$, depending on
$z,\delta,w$ but not on $\zeta$, such that
$\zeta\pm te^{i\theta}w\in D$ for every
$\zeta\in B(z,\tau_\delta)$ whenever $|t|<t_0$. For such $t$,
Proposition~\ref{prop:second-difference} applies at almost every
$\zeta\in B(z,\tau_\delta)$. Averaging with
$\chi_{\tau_\delta}(z-\zeta)$ gives
\begin{align*}
 &h_\delta(z+t e^{i\theta}w)+h_\delta(z-t e^{i\theta}w)-2h_\delta(z)\\
 &\qquad\le
 2t^2\left(\frac{|\partial h(w)|^2}{h}*\chi_{\tau_\delta}\right)(z).
\end{align*}
Divide by $t^2$, let $t\to0$, and use
\eqref{eq:complex-line-second}. We obtain
\[
 L_{h_\delta}(z;w)+\operatorname{Re}\bigl(e^{2i\theta}Q_{h_\delta}(z;w)\bigr)
 \le\left(\frac{|\partial h(w)|^2}{h}*\chi_{\tau_\delta}\right)(z).
\]
Taking the maximum over $\theta$ yields
\begin{equation}\label{eq:convineq}
 L_{h_\delta}(z;w)+|Q_{h_\delta}(z;w)|
 \le\left(\frac{|\partial h(w)|^2}{h}*\chi_{\tau_\delta}\right)(z).
\end{equation}

It remains to compare the right-hand side with
$|\partial h_\delta(z)(w)|^2/h_\delta(z)$. Expanding the square gives the
exact identity
\begin{align}\label{eq:variance-identity}
 &\int
 \frac{|\partial h(\zeta)(w)|^2}{h(\zeta)}
 \chi_{\tau_\delta}(z-\zeta)\,d\zeta
 -\frac{|\partial h_\delta(z)(w)|^2}{h_\delta(z)}\notag\\
 &\qquad=
 \int h(\zeta)
 \left|
 \frac{\partial h(\zeta)(w)}{h(\zeta)}
 -\frac{\partial h_\delta(z)(w)}{h_\delta(z)}
 \right|^2
 \chi_{\tau_\delta}(z-\zeta)\,d\zeta.
\end{align}

If $\zeta\in B(z,\tau_\delta)$, then for small $\delta$
\[
 \frac{3\delta}{4}<\delta_D(\zeta)<\frac{17\delta}{4},
\]
so
\begin{equation}\label{eq:hexplicitbounds}
 \frac{9}{16}\delta^2<h(\zeta)<\frac{289}{16}\delta^2
\end{equation}
and, at differentiability points,
\[
 |\partial h(\zeta)(w)|\le\frac{17}{4}\delta|w|.
\]
For $\zeta,\eta\in B(z,\tau_\delta)$ we have
$|\zeta-\eta|\le2\tau_\delta$. For sufficiently small $\delta$,
$2\tau_\delta\le\delta/4$, so Lemma~\ref{lem:gradosc} gives
\[
 |\partial h(\zeta)(w)-\partial h(\eta)(w)|
 \le C\delta^\gamma|w|
\]
for almost every pair $\zeta,\eta$. Also
\eqref{eq:hosc} gives
\[
 |h(\zeta)-h(\eta)|\le C\delta\tau_\delta.
\]
Consequently
\begin{align*}
 \left|
 \frac{\partial h(\zeta)(w)}{h(\zeta)}
 -\frac{\partial h(\eta)(w)}{h(\eta)}
 \right|
 &\le
 \frac{|\partial h(\zeta)(w)-\partial h(\eta)(w)|}{h(\zeta)}\\
 &\quad+
 \frac{|\partial h(\eta)(w)|\,|h(\zeta)-h(\eta)|}
 {h(\zeta)h(\eta)}\\
 &\le C\bigl(\delta^{\gamma-2}+\tau_\delta\delta^{-2}\bigr)|w|\\
 &\le C\delta^{\gamma-2}|w|.
\end{align*}
Now
\[
 \frac{\partial h_\delta(z)(w)}{h_\delta(z)}
 =\frac{1}{h_\delta(z)}
 \int h(\eta)\frac{\partial h(\eta)(w)}{h(\eta)}
 \chi_{\tau_\delta}(z-\eta)\,d\eta,
\]
so it is the average of $\partial h(\eta)(w)/h(\eta)$ with respect to the
probability measure
\[
 \frac{h(\eta)\chi_{\tau_\delta}(z-\eta)}{h_\delta(z)}\,d\eta.
\]
Therefore the same oscillation bound gives
\[
 \left|
 \frac{\partial h(\zeta)(w)}{h(\zeta)}
 -\frac{\partial h_\delta(z)(w)}{h_\delta(z)}
 \right|
 \le C\delta^{\gamma-2}|w|
\]
for almost every $\zeta\in B(z,\tau_\delta)$. Using
\eqref{eq:hexplicitbounds} in \eqref{eq:variance-identity}, we obtain
\[
 \left(\frac{|\partial h(w)|^2}{h}*\chi_{\tau_\delta}\right)(z)
 -\frac{|\partial h_\delta(z)(w)|^2}{h_\delta(z)}
 \le C\delta^{2\gamma-2}|w|^2.
\]
Combining this with \eqref{eq:convineq}, and using
$2\gamma-2=\frac{2\alpha}{1-\alpha}$, proves
\eqref{eq:mollifiedineq}.

Finally, if $\zeta\in B(z,\tau_\delta)$, then \eqref{eq:hosc} gives
\[
 |h(\zeta)-h(z)|\le10\delta|\zeta-z|\le10\delta\tau_\delta.
\]
Averaging proves \eqref{eq:C0approx}.
\end{proof}

\begin{lemma}\label{lem:strictification}
Let $U\subset B(0,R)$, let $g\in C^2(U)$ be positive, and suppose
\[
 \widetilde g=g-\varepsilon(1+|z|^2)>0.
\]
Then, for every $w\in\C^n$,
\begin{align}\label{eq:strictification}
 &\frac{|\partial\widetilde g(w)|^2}{\widetilde g}
 -L_{\widetilde g}(w)-|Q_{\widetilde g}(w)|\\
 &\qquad\ge
 \frac{|\partial g(w)|^2}{g}-L_g(w)-|Q_g(w)|
 +\frac{\varepsilon}{1+R^2}|w|^2.\nonumber
\end{align}
\end{lemma}

\begin{proof}
Note that
\[
 \partial(1+|z|^2)(w)=\langle w,z\rangle,
 \qquad
 L_{1+|z|^2}(w)=|w|^2,
 \qquad
 Q_{1+|z|^2}(w)=0.
\]
A direct calculation gives
\[
\begin{aligned}
 &\frac{|\partial g(w)-\varepsilon\langle w,z\rangle|^2}
 {g-\varepsilon(1+|z|^2)}
 -\frac{|\partial g(w)|^2}{g}
 +\varepsilon\frac{|\langle w,z\rangle|^2}{1+|z|^2}\\
 &\qquad=
 \frac{\varepsilon\left|(1+|z|^2)\partial g(w)-g\langle w,z\rangle\right|^2}
 {g\bigl(g-\varepsilon(1+|z|^2)\bigr)(1+|z|^2)}\ge0.
\end{aligned}
\]
Therefore
\begin{align*}
 &\left(\frac{|\partial\widetilde g(w)|^2}{\widetilde g}
 -L_{\widetilde g}(w)-|Q_{\widetilde g}(w)|\right)\\
 &\quad-\left(\frac{|\partial g(w)|^2}{g}
 -L_g(w)-|Q_g(w)|\right)\\
 &\qquad\ge\varepsilon\left(
 |w|^2-\frac{|\langle w,z\rangle|^2}{1+|z|^2}\right).
\end{align*}
Since $|\langle w,z\rangle|\le|w||z|$,
\[
 |w|^2-\frac{|\langle w,z\rangle|^2}{1+|z|^2}
 \ge\frac{|w|^2}{1+|z|^2}
 \ge\frac{|w|^2}{1+R^2}.
\]
This proves \eqref{eq:strictification}.
\end{proof}

For $z\in U_\delta$ and $\zeta\in B(z,\tau_\delta)$ we have
$\delta_D(\zeta)>\delta-\tau_\delta$, and therefore
\[
 h_\delta(z)\ge(\delta-\tau_\delta)^2.
\]
Thus, for all sufficiently small $\delta$, $\widetilde h_\delta>0$ on
$U_\delta$.

\begin{proposition}\label{prop:strict}
There is $c_*>0$ such that, for all sufficiently small $\delta>0$,
every $z\in U_\delta$ and every $w\in\C^n$ satisfy
\begin{equation}\label{eq:strictineq}
 \frac{|\partial\widetilde h_\delta(z)(w)|^2}{\widetilde h_\delta(z)}
 -L_{\widetilde h_\delta}(z;w)-|Q_{\widetilde h_\delta}(z;w)|
 \ge c_*\delta^2|w|^2.
\end{equation}
\end{proposition}

\begin{proof}
Proposition~\ref{prop:mollification} and
Lemma~\ref{lem:strictification}, with $\varepsilon=c\delta^2$, give
\[
 \frac{|\partial\widetilde h_\delta(w)|^2}{\widetilde h_\delta}
 -L_{\widetilde h_\delta}(w)-|Q_{\widetilde h_\delta}(w)|
 \ge
 \left(\frac{c}{1+R^2}\delta^2
 -C\delta^{\frac{2\alpha}{1-\alpha}}\right)|w|^2.
\]
Since $2\alpha/(1-\alpha)>2$, for all sufficiently small $\delta$ we have
\[
 C\delta^{\frac{2\alpha}{1-\alpha}}
 \le \frac{c}{2(1+R^2)}\delta^2.
\]
Thus \eqref{eq:strictineq} holds with
$c_*=c/(2(1+R^2))$.
\end{proof}

\begin{proof}[Completion of the proof of Theorem~\ref{thm:main}]
We now verify the properties of the domains defined in \eqref{eq:Vdelta}.
Since $\delta_D$ is $1$-Lipschitz, whenever $\delta_D(z)>\tau_\delta$,
\begin{equation}\label{eq:basicconvbounds}
 (\delta_D(z)-\tau_\delta)^2\le h_\delta(z)
 \le(\delta_D(z)+\tau_\delta)^2.
\end{equation}
Consequently, after decreasing $\delta_0>0$ if necessary, for every
$0<\delta<\delta_0$ we have
\[
 \delta_D(z)\le\frac{3\delta}{2}
 \quad\Longrightarrow\quad
 \widetilde h_\delta(z)<\frac92\delta^2,
\]
whereas
\[
 \delta_D(z)\ge\frac{5\delta}{2}
 \quad\Longrightarrow\quad
 \widetilde h_\delta(z)>\frac92\delta^2.
\]
Indeed, both implications follow directly from \eqref{eq:basicconvbounds},
$\tau_\delta/\delta=\delta^{\gamma-1}\to0$, and the fixed choice of $c$.
In particular,
\begin{equation}\label{eq:corecontainment}
 \left\{\delta_D>\frac{5\delta}{2}\right\}
 \subset V_\delta
 \subset
 \left\{\delta_D>\frac{3\delta}{2}\right\}.
\end{equation}

Choose $\delta_0$ still smaller so that Proposition~\ref{prop:strict}
applies for $0<\delta<\delta_0$ and
$\delta_D(z_*)>5\delta_0/2$. Then $z_*\in V_\delta$ for every
$0<\delta<\delta_0$.

Let $p\in\partial G_\delta$. By \eqref{eq:corecontainment},
\[
 \frac{3\delta}{2}<\delta_D(p)<\frac{5\delta}{2},
 \qquad
 \widetilde h_\delta(p)=\frac92\delta^2.
\]
In particular $p\in U_\delta$. We claim that
$\partial\widetilde h_\delta(p)\ne0$. Otherwise
Proposition~\ref{prop:strict} gives
\[
 -L_{\widetilde h_\delta}(p;w)-|Q_{\widetilde h_\delta}(p;w)|
 \ge c_*\delta^2|w|^2
\]
for every $w\in\C^n$. By \eqref{eq:complex-line-second}, the real Hessian
of $\widetilde h_\delta$ at $p$ is then negative definite. Thus $p$ is a
strict local maximum of $\widetilde h_\delta$, contradicting the fact that
$p\in\partial V_\delta$ and $V_\delta$ is locally the superlevel set
$\{\widetilde h_\delta>9\delta^2/2\}$. Hence $\partial G_\delta$ is
$C^\infty$.

Near $p\in\partial G_\delta$ a defining function is
\[
 r=\frac92\delta^2-\widetilde h_\delta.
\]
Let $w\ne0$ be complex tangent at $p$. Then
$\partial\widetilde h_\delta(p)(w)=0$, and
Proposition~\ref{prop:strict} yields
\[
 -L_{\widetilde h_\delta}(p;w)-|Q_{\widetilde h_\delta}(p;w)|>0.
\]
Since $L_r=-L_{\widetilde h_\delta}$ and
$Q_r=-Q_{\widetilde h_\delta}$,
\[
 L_r(p;w)>|Q_r(p;w)|.
\]
Thus $G_\delta$ is strongly linearly convex. By
\eqref{eq:corecontainment}, $G_\delta\Subset D$.
Now set $\delta_j=2^{-j}\delta_0$, $D_j=G_{\delta_j}$.
From \eqref{eq:corecontainment},
\[
 \overline{D_j}\subset
 \left\{\delta_D\ge\frac{3\delta_j}{2}\right\}
 \subset
 \left\{\delta_D>\frac{5\delta_{j+1}}2\right\}
 \subset V_{\delta_{j+1}}.
\]
Since $\overline{D_j}$ is connected and contains $z_*$, it lies in the
component $G_{\delta_{j+1}}=D_{j+1}$. Hence
$\overline{D_j}\subset D_{j+1}$.

Finally, let $z\in D$ and join $z_*$ to $z$ by a continuous path in $D$.
The image of the path is compact, hence has positive distance $m$ from
$\partial D$. For all sufficiently large $j$, $5\delta_j/2<m$, so the
whole path lies in
$\{\delta_D>5\delta_j/2\}\subset V_{\delta_j}$. Therefore $z$ belongs to
the same component of $V_{\delta_j}$ as $z_*$, namely $D_j$. Thus $\bigcup_{j=1}^\infty D_j=D$.
This completes the proof.
\end{proof}

\end{document}